\documentclass[11pt,centertags]{amsart}

\usepackage{amsmath}
\usepackage{amsthm}
\usepackage{amssymb}
\usepackage{amstext}
\usepackage{amscd}
\usepackage{mathrsfs} 
\usepackage[pdftex,a4paper]{hyperref} 
\usepackage{exscale}
\usepackage{verbatim}
\usepackage[arrow, matrix, curve]{xy}
\usepackage{enumerate}
\usepackage{graphics}

\usepackage{fouriernc}
\usepackage[T1]{fontenc}

\numberwithin{equation}{section}

\allowdisplaybreaks

\input xy
\xyoption{all}
\theoremstyle{definition}
	\newtheorem{definition}{Definition} 
	\newtheorem*{definition*}{Definition}
	
	\numberwithin{definition}{section}
\theoremstyle{plain}
	\newtheorem{lemma}[definition]{Lemma}
	
	\newtheorem{proposition}[definition]{Proposition}
	\newtheorem{theorem}[definition]{Theorem}
	\newtheorem*{theorem*}{Theorem}
	\newtheorem{corollary}[definition]{Corollary}

	\newtheorem*{claim*}{Claim}
\theoremstyle{remark}
	
	\newtheorem{question}[definition]{Question}

	\newtheorem{remark}[definition]{Remark}

\DeclareMathOperator{\Ad}{Ad}
\DeclareMathOperator{\ad}{ad}

\renewcommand{\bf}{\textbf}

\renewcommand{\phi}{\varphi}

\newcommand{\C}{\mathbb{C}}

\newcommand{\R}{\mathbb{R}}

\newcommand{\gfr}{\mathfrak{g}}

\newcommand{\hfr}{\mathfrak{h}}
\newcommand{\kfr}{\mathfrak{k}}

\DeclareMathOperator{\im}{im}

\newcommand{\bR}{\mathbb{R}}

\renewcommand\phi{\ensuremath{\varphi}}
\renewcommand\epsilon{\ensuremath{\varepsilon}}

\renewcommand{\L}{\mathfrak}
\newcommand{\ra}{\to}

\numberwithin{equation}{section}

\renewcommand{\emph}{\textbf}

\numberwithin{equation}{section}

\renewcommand{\L}{\mathfrak}
\renewcommand{\ra}{\to}

\renewcommand{\L}{\mathfrak}

\begin{document}

\title[Quasi-state rigidity]{Quasi-state rigidity for finite-dimensional Lie algebras}
\author{Michael Bj\"orklund and Tobias Hartnick}
\date{\today}
\begin{abstract}
We say that a Lie algebra $\gfr$ is quasi-state rigid if every Ad-invariant continuous Lie quasi-state on it is the directional derivative of a homogeneous quasimorphism. Extending work of Entov and Polterovich, we show that every reductive Lie algebra, as well as the algebras $\C^n \rtimes \L{u}(n)$, $n \geq 1$, are rigid. On the other hand, a Lie algebra which surjects onto the three-dimensional Heisenberg algebra is not rigid. For Lie algebras of dimension $\leq 3$ and for solvable Lie algebras which split over a codimension one abelian ideal, we show that this is the only obstruction to rigidity.
\end{abstract}

\maketitle


\section{Introduction}

\subsection{Background and history}
Let $\L g$ be a real Lie algebra. A function $\zeta: \gfr \to \R$ is called a \emph{Lie quasi-state} if 
\[
\zeta(aX + bY) = a\zeta(X) + b\zeta(Y).
\]
for all $a, b \in \R$ and every pair $(X, Y)$ of \emph{commuting} elements in $\gfr$. In other words, $\zeta$ is linear on abelian subalgebras of $\gfr$. 

We refer the reader to \cite{E14, EP09} for an overview of the history of the notion of Lie quasi-states. Roughly speaking, Lie quasi-states (or closely related notions) arose more or less independently in three different contexts: In connection with the foundations of quantum mechanics (see e.g. \cite{Gl57, Gleason2, Gleason3}), in symplectic topology (see the recent survey \cite{E14} for references) and in the study of quasimorphisms on finite (see e.g. \cite{GuichardetWigner, BuMo, Shtern, Surface, BSH, BSH2, Calegari}) and infinite-dimensional (see e.g. \cite{EP1, EPBiran, BS, Usher}) Lie groups.

One of the basic theorems in the mathematical foundations of quantum mechanics is Gleason's theorem on rigidity of frame functions \cite{Gl57, Gleason2}. Although Lie quasi-states do not feature explicitly in Gleason' work, his result is essentially equivalent to the statement that every locally bounded Lie quasi-state on $\L u(n)$ is linear, provided $n \geq 3$. (See the introductions of \cite{EPZ} and \cite{EP09} for a discussion of this equivalence.) This can be seen as the first major non-trivial result concerning Lie quasi-states.

In symplectic topology, Lie quasi-states constructed from Floer homology and spectral invariants (as in \cite{EP1, EPBiran, Usher}) have recently become an important tool in studying the displacability of subsets of symplectic manifolds under Hamiltonian diffeomorphisms (see \cite{EP2}). Entov's recent ICM address \cite{E14} gives an overview of these developments and an extensive list of references. In most of these symplectic applications, the Lie quasi-states considered arise as directional derivatives of continuous quasimorphisms on the corresponding infinite-dimensional Lie groups.

Both in the quantum-mechanical and the symplectic setting the focus is naturally on Lie quasi-states on infinite-dimensional Lie algebras. A systematic analysis of Lie quasi-states on finite-dimensional Lie algebras was initiated only recently by Entov and Polterovich in \cite{EP09}. Even if one is primarily interested in the infinite-dimensional case, such an analysis is relevant in order to understand the behaviour of Lie quasi-states along finite-dimensional subalgebras. However, to the best of our knowledge \cite{EP09} remains the only paper so far which concerns Lie quasi-states on finite-dimensional Lie algebras. 

The purpose of the present paper is to extend some of the results from \cite{EP09} to larger classes of finite-dimensional Lie algebras and to obtain a clearer picture about Lie quasi-states on general finite-dimensional Lie algebras through some key examples. Our main focus will be on Ad-invariant Lie quasi-states, since these are comparatively easy to handle and at the same time the most relevant ones in applications. One particular goal of this article is to understand their connection with homogeneous quasimorphisms on finite-dimensional Lie groups.

\subsection{Integrable Lie quasi-states and homogeneous quasimorphisms}
From now on we assume that all Lie algebras are real and finite-dimensional. We denote by $Q(\gfr)$ the space of Lie quasi-states on $\gfr$ and by $\mathcal Q(\gfr) \subset Q(\gfr)$ the subspace of continuous Lie quasi-states on $\gfr$ respectively. Note that the adjoint group associated to the Lie algebra $\gfr$ acts on these spaces by $g.\zeta(X) = \zeta({\rm Ad}(g)^{-1}(X))$.

The notion of a Lie quasi-state on a Lie algebra $\L g$ has a global counterpart on the level of Lie groups. Indeed, given a Lie group $G$ (not necessarily connected) we can consider \emph{integrated Lie quasi-states}, i.e. Borel measurable functions $f: G \to \R$ such that
\begin{equation*}
f(gh) = f(g) + f(h).
\end{equation*}
for all pairs $(g,h)$ of commuting elements in $G$. Any integrated Lie quasi-state $f$ on a Lie group $G$ induces a Lie quasi-state $\zeta$ on the Lie algebra $\gfr$ on $G$, given by the \emph{directional derivative}
\[
\zeta(X) =f(\exp_G(X)), \quad \textrm{for all $X \in \gfr$},
\]
where $\exp_G$ denotes the exponential map from $\gfr$ to $G$. 
We note that $\zeta$ is continuous provided $f$ is. If $\zeta$ is the directional derivative of an integrated Lie quasi-state $f$, then we say that $\zeta$ is \emph{integrable} and we say that $\zeta$ \emph{integrates} to $f$. Typical examples of integrated Lie quasi-states are given by homogeneous (continuous) quasimorphisms, i.e. continuous functions $f: G\to \R$ which satisfy
\[
D(f) := \sup_{g,h \in G} |f(gh)-f(g)-f(h)| < \infty, 
\]
and $f(g^n) = n \cdot f(g)$ for all $n \in \mathbb Z$. We note that if $g$ and 
$h$ in $G$ commute, then 
\[
n \cdot f(gh) = f((gh)^n) = f(g^n h^n) - f(g^n) - f(h^n) + n \cdot (f(g) + f(h))
\]
for all $n$, which readily implies that $f(gh) = f(g) + f(h)$ upon dividing with 
$n$ and letting $n$ tend to infinity. In particular, every homogeneous (continuous) quasimorphism is a conjugation-invariant integrated Lie quasi-state on $G$, and thus its directional derivative  gives rise to an {\rm Ad}-invariant 
(continuous) Lie quasi-states on $\gfr$.  

It turns out that homogeneous quasimorphisms on connected Lie groups are rare. On a solvable Lie group every homogeneous quasimorphism  is in fact a homomorphism, and a simple Lie group admits a (non-trivial) homogeneous quasimorphism if and only if it has infinite center, in which case there is a unique (non-trivial) homogeneous quasimorphism up 
to multiples (see e.g. \cite{BSH2}).

\subsection{Lie quasi-state rigidity}

In the sequel we shall denote by $Q_{\rm Ad}(\gfr)$ and $Q_{\rm int}(\gfr)$ the spaces of Ad-invariant and integrable Lie quasi-states respectively. We also denote by $Q_{\rm qm}(\gfr)$ the space of Lie quasi-states which are directional derivatives of homogeneous quasimorphisms on some Lie group $G$ with Lie algebra $\gfr$. We then use the notations $\mathcal Q_{\rm Ad}(\gfr)$, $\mathcal Q_{\rm int}(\gfr)$, $\mathcal Q_{\rm qm}(\gfr)$ for the corresponding subspaces of continuous quasi-states. We write $\gfr^*$ for the space of linear functionals on $\gfr$ and ${\rm Hom}(\gfr, \R) \subset \gfr^*$ for the subspace of Lie algebra homomorphisms from $\gfr$ to $\R$. Then we have a chain of inclusions
\[
{\rm Hom}(\gfr, \R) \subset Q_{\rm qm}(\gfr) = \mathcal Q_{\rm qm}(\gfr) \subset   \mathcal Q_{\rm Ad}(\gfr) \cap  \mathcal  Q_{\rm int}(\gfr) \subset \mathcal Q_{\rm Ad}(\gfr) \subset \mathcal Q(\gfr).
\]
It follows from the results in \cite{BuMo} that the first inclusion is an equality for all Lie algebras not containing a simple Hermitian Lie algebra, and it is of codimension one for simple Hermitian Lie algebras (see also \cite{BSH} for a more detailed discussion). In particular, $Q_{\rm qm}(\gfr)$ is always finite-dimensional, whereas all of the larger spaces listed above can be infinite-dimensional. We say that a Lie algebra $\gfr$ is \emph{quasi-state rigid}, or just \emph{rigid} for short, if 
\[
\mathcal Q_{\rm qm}(\gfr) =\mathcal Q_{\rm Ad}(\gfr).
\]
This implies in particular, that every {\rm Ad}-invariant Lie quasi-state is integrable and that $\mathcal Q_{\rm Ad}(\gfr)$ is finite-dimensional. For Lie algebras not containing a simple Hermitian Lie algebra it is moreover equivalent to showing that every {\rm Ad}-invariant Lie quasi-state is linear, hence a homomorphism. In the sequel we will refer to the problem of classifying all rigid Lie algebras as the \emph{(quasi-state)} \emph{rigidity problem}, and this problem will be the main concern of the present paper.

\subsection{Statements of main results}
We recall that every finite-dimensional Lie algebra $\gfr$ is the semidirect product of a semisimple Lie subalgebra and a maximal solvable ideal, called the solvable radical of $\gfr$. It is therefore a common strategy in the study of Lie algebras to separately analyze a problem for semisimple (or, slightly more generally, reductive) and solvable Lie algebras and then to attack the general `mixed' case by means of semi-direct products. This is also the approach which we will take towards the quasi-state rigidity problem here.

The reductive case is by far the simplest one, since the fine structure of reductive Lie algebras is very well understood. It was already established in \cite[Theorem 4.2]{EP09} that  for every Hermitian simple Lie algebra $\gfr$ we have 
\[
Q_{\rm Ad}(\L g) = {\mathcal Q}_{\rm Ad}(\L g)
={\mathcal Q}_{\rm qm}(\L g).
\] 
By  \cite[Theorem 4.1]{EP09} the same holds for any compact Lie algebra $\gfr$ (which is automatically reductive). In this paper, we strengthen this 
result as follows:
\begin{theorem}\label{MainTheorem1}
For every reductive Lie algebra $\gfr$ we have 
\[Q_{\rm Ad}(\L g) = {\mathcal Q}_{\rm Ad}(\L g)
={\mathcal Q}_{\rm qm}(\L g).\]
In particular, every reductive Lie algebra is rigid and every Ad-invariant Lie quasi-state on a non-Hermitian simple Lie algebra (such as $\L {sl}_n(\R)$, $n \geq 3$) is trivial.
 \end{theorem}
This settles the quasi-state rigidity problem in the reductive setting completely. However, we stress that the classification of 
\textit{non}-Ad-invariant quasi-states on reductive Lie algebras 
remains open. There are some important partial results towards such a classification. For example, if $n \geq 3$ and $\gfr$ denotes either the Lie algebra $\gfr = \L u(n)$ or the Lie algebra $\gfr = \L {sp}(2n)$, then 
\begin{equation}\label{EPGl}
{\mathcal Q}(\L g) ={\mathcal Q}_{\rm qm}(\L g)+\gfr^*,
\end{equation}
where in the $\L u(n)$ case the first summand on the right hand side is actually trivial. For $\gfr = \L u(n)$ this follows from a classical theorem of Gleason (\cite{Gl57}, see also the introduction of 
\cite{EP09}) and for $\gfr = \L {sp}(2n)$ this is one of the main results of \cite{EP09}. Currently we do not know whether \eqref{EPGl} can be extended to more general compact Lie algebras, let alone reductive Lie algebras. If such a general result exists, then it has to evoke some kind of higher rank assumption, since \eqref{EPGl} fails for $\L{sp}(2)$ and $\L u(2)$.

Let us mention in passing that there is also a global version of Theorem \ref{MainTheorem1} which can be stated as follows (see Theorem \ref{GlobalThm} below):
\begin{theorem} Let $G$ be a connected reductive Lie group. Then every conjugation-invariant integrated Lie quasi-state on $G$ is a homogeneous quasimorphism.
\end{theorem}

The rigidity problem for solvable Lie algebras is more subtle than in the reductive case. The smallest example of a non-rigid solvable Lie algebra is provided by the three-dimensional Heisenberg algebra $\L h_3$, which can be characterized as the unique non-rigid Lie algebra of dimension $\leq 3$ (see Theorem \ref{ThmLow}). We actually show that $\mathcal Q_{Ad}(\L h_3)$ is infinite-dimensional by providing an explicit parametrization of all Ad-invariant continuous Lie quasi-states on $\L h_3$ (see Corollary \ref{CorH3}). It then follows from general principles that $\mathcal Q_{Ad}(\L g)$ is also infinite-dimensional for every Lie algebra $\gfr$ which surjects onto $\L h_3$. In particular, such a Lie algebra can never be rigid. One can ask whether this is the only obstruction to rigidity, at least for solvable Lie algebras:

\begin{question}\label{Question} Let $\gfr$ be a solvable Lie algebra which does not surject onto $\L h_3$. Does it follow that $\gfr$ is rigid?
\end{question}

This question is motivated by the following partial results. Consider the class of Lie algebras $\gfr$ which are almost abelian in the in the sense that $\gfr$ admits a codimension one abelian ideal $V$ such that the extension
\[
0 \to V \to \gfr \to \R \to 0
\]
splits as a semidirect product. Such Lie algebras are automatically (two-step) solvable, but not necessarily nilpotent. Examples include the Lie algebras of the $(ax+b)$-group, of the three-dimensional SOL group and of the $1$-dimensional unitary motion group $\C \rtimes U(1)$, as well as the Lie algebra $\L h_3$. For Lie algebras in this class, the answer to Question \ref{Question} is positive:
\begin{theorem}\label{MainTheorem2}
Let $\gfr$ be a Lie algebra which splits over a codimension one abelian ideal. Then $\gfr$ is rigid if and only if it does not surject onto the three-dimensional Heisenberg algebra $\L h_3$.
\end{theorem}
Given the special role of the three-dimensional Heisenberg algebra in this theorem, it is natural to ask about rigidity of the higher-dimensional Heisenberg algebras. Interestingly enough, it was already established in \cite[Prop. 2.13]{EP09} that these algebras are always rigid. In fact, on these algebras every Lie quasi-state (neither assumed Ad-invariant nor even continuous) is linear. Thus in order to find potential counterexamples to Question \ref{Question} one has to look elsewhere (and to go beyond dimension $3$), and it it currently not clear to us, what would be a natural class of solvable Lie algebras to consider.

The example of the Heisenberg algebra should also serve as a warning concerning the study of the rigidity problem for general (i.e. neither solvable nor reductive) Lie algebras. Here one is immediately confronted with the following problem:
\begin{question}\label{Question2}
Let $\L g_1$ be a rigid solvable Lie algebra and $\L g_2$ be a semisimple (hence rigid) Lie algebra acting on $\L g_1$ by automorphisms. Is the semi-direct product $\L g_1 \rtimes \L g_2$ rigid?
\end{question}
If the assumption semisimple is replaced by reductive the answer to this question is no, since $\L h_3 = \R^2 \rtimes \R$ is even a semi-direct product of abelian subalgebras. For the moment we have no systematic way to deal with Question \ref{Question2}.
However, we do understand some specific, but important, examples:

\begin{theorem}\label{MainTheorem3} Let $\gfr_n$ be the $n$-dimensional unitary motion algebra, i.e. the semidirect product of the Lie algebras $\L{u}(n)$ and $\C^n$ with respect to the standard representation of $\L{u}(n)$ on $\C^n$. Then the following hold:
\begin{enumerate}[(i)]
\item $\gfr_n$ is rigid for all $n \geq 1$.
\item If $n \geq 3$, then every every continuous quasi-state on $\gfr$ is linear.
\item If $n \leq 2$, then the space of continuous quasi-states on $\gfr_n$ is infinite-dimensional.
\end{enumerate}
\end{theorem}
The proof of Theorem \ref{MainTheorem3} is close in spirit to the proofs 
of Gleason and of Entov-Polterovich of the rigidity of the Lie algebras 
$\L u(n)$ and $\L{sp}(n)$ respectively in that it depends on the analysis of
the values of a given Lie quasi-state at certain elements related to rank one projections. Extending Theorem \ref{MainTheorem3} even to the Euclidean motion algebras $\R^n \rtimes \L{o}(n)$ would require a deeper
understanding of (non-Ad-invariant) Lie quasi-states on $\L{o}(n)$.

\subsection{Organization of the paper}
The paper consists of four main sections and an appendix. In Section \ref{sec:reductive} we discuss additive Jordan decompositions 
of simple Lie algebras and we show how these decompositions can be used 
to prove Theorem \ref{MainTheorem1}. In Section \ref{sec:semidirect} we provide an explicit formula for Lie quasi-states on solvable Lie algebras which split over an abelian ideal of 
codimension $1$. We then use this formula to derive Theorem \ref{MainTheorem2}. 
In Section \ref{sec:unitary} we extend some arguments of Gleason and
Entov-Polterovich to the setting of unitary motion Lie algebras and 
establish Theorem \ref{MainTheorem3}. In the  appendix we study a class of generalized frame functions which appear in the proof of Theorem \ref{MainTheorem3}

\subsection{Acknowledgments}
The authors would like to express their gratitude to Michael Entov
for bringing the problem of classifying Lie quasi-states on finite-dimensional 
Lie algebras to their attention. They are also indebted to Leonid Polterovich for comments on an earlier version of this paper and in particular for challenging them to extend Theorem \ref{MainTheorem3} to the present form. The authors are grateful to the anonymous referee for a careful reading of the manuscript which lead to several improvements and corrections. In particular, the referee pointed out a crucial mistake in one of the computations in an earlier version of this paper, and thereby saved us from publishing a wrong classification statement. The present paper is an outgrowth of discussions between the authors which took place at Chalmers University of Technology in Gothenburg in September 2014. The second author would 
like to thank the Guest Research Program of Chalmers and the Department of Mathematical Sciences at Chalmers for their hospitality. 

\section{The reductive case}
\label{sec:reductive}
Recall that a Lie algebra $\gfr$ is \emph{reductive} if it decomposes as a direct sum of abelian and simple Lie algebras. The study of (Ad-invariant) Lie quasi-states on such Lie algebras can be immediately reduced to the case of simple Lie algebras by means of the following observation:
\begin{lemma}\label{Trivialities} If $\L g = \L g_1 \oplus \L g_2$ is a direct sum of Lie algebras, then $Q(g) $ decomposes as $ Q(\gfr) =Q(\gfr_1) \oplus Q(\gfr_2)$, and the subspaces ${Q}_{\rm qm}(\L g)$, ${Q}_{\rm Ad}(\L g)$, ${Q}(\L g)$, ${\mathcal Q}_{\rm qm}(\L g)$, ${\mathcal Q}_{\rm Ad}(\L g)$, ${\mathcal Q}(\L g)$ decompose accordingly. If $\gfr$ is abelian then all these spaces coincide with $\gfr^*$.
\end{lemma}

\subsection{Additive Jordan decompositions of simple Lie algebras}

Thus in order to establish Theorem \ref{MainTheorem1} it suffices to consider the case of a finite-dimensional simple real Lie algebra $\gfr$. Such Lie algebras come in two different flavors, \emph{Hermitian} and \emph{non-Hermitian}. More precisely, let $\theta$ be a Cartan involution on $\gfr$ and let $\kfr<\gfr$ be the fixed point algebra of $\theta$. Then $\kfr$ is reductive, i.e. can be written as $\kfr = \L {z(k)} \oplus \kfr_{ss}$ where $\kfr_{ss}$ is semisimple and where $\L{z(k)}$ denotes the centre of $\kfr$, and moreover $\dim \L{z(k)} \leq 1$. Now $\gfr$ is Hermitian if $\dim \L{z(k)} =1$ and non-Hermitian otherwise. 

A key feature of simple Lie algebras is that they are always algebraic, i.e. their associated adjoint groups are not only Lie groups, but in fact algebraic groups. Now simple algebraic groups admit a (multiplicative) Jordan decomposition, and this induces an additive Jordan decomposition of the corresponding Lie algebras. We will now describe this decomposition explicitly in our setting.

We first recall that given the pair $(\gfr, \kfr)$, there exists an Iwasawa decomposition of the form 
\[
\gfr = \kfr \oplus \L a \oplus \L n =  \L {z(k)} \oplus \kfr_{ss} \oplus \L a \oplus \L n,
\]
such that $\L a < \L g$ is an abelian subalgebra with $\theta|_{\L a} \equiv -1$ and maximal with respect to these two properties and $\L n$ consists of {\rm ad}-nilpotent elements \cite[Prop. 6.43]{Knapp}. Now the additive refined Jordan decomposition of $\gfr$ can be stated as follows \cite[Cor. 2.5]{Borel}:
\begin{lemma}\label{AdditiveRJC} Let $\gfr$ be a simple Lie algebra with Iwasawa decomposition $\gfr =\L {z(k)} \oplus \kfr_{ss} \oplus \L a \oplus \L n$ and let $G$ be the $1$-connected Lie group with Lie algebra $\gfr$. Then for every $X \in \gfr$ there exist unique elements $X_c, X_k, X_a, X_n \in \gfr$ and (in general non-unique) elements $Y_c \in \L{z(k)}, Y_k \in \kfr_{ss}$, $Y_a \in \L a$, $Y_n \in \L n$ such that the following hold:
\begin{enumerate}[(i)]
\item $X = X_c+X_k + X_a + X_n$.
\item $X_c, X_k, X_a, X_n$ commute pairwise.
\item $X_j \in {\rm Ad(G)}(Y_j)$ for $j \in \{c, k,a,n\}$.
\end{enumerate}
\end{lemma}
Note that the elements $Y_c, Y_k, Y_a, Y_n$ are determined up to the action of ${\rm Ad}(G)$. In fact, the element $Y_c$ is uniquely determined, as can be seen as follows: Assume $Y_c, Y_c'$ were elements of $\L{z(k)}$ and ${\rm Ad}(g)(Y_c) = Y_c'$. Since $\dim \L{z(k)} \leq 1$ we have $Y_c = \lambda Y_c'$ for some $\lambda \in \R$. In particular, if $f: G\to \R$ is conjugation-invariant then $f(\exp(Y_c)) = f(\exp(Y_c'))$. If $\lambda \neq 1$ we would deduce that every continuous conjugation-invariant homogeneous function on $G$ would have to vanish on $Z(K)$; however, a non-trivial such function is given by the Guichardet-Wigner quasimorphism \cite{GuichardetWigner}. We deduce that $Y_c$ is uniquely determined by $X$ and refer to $Y_c$ as the \emph{central elliptic part} of $X$. 

We need one other basic structural property of semisimple Lie algebras. Recall from \cite[Prop. 14.31]{FultonHarris} that given an irreducible abstract root system $\Sigma$ spanning a vector space $V$, the action of the Weyl group $W$ of $\Sigma$ on $V$ is irreducible. It follows that the space $V^W$ of $W$-invariants is trivial, and this conclusion extends to reducible root systems. In particular, if $\L a$ is the Cartan subalgebra of a semisimple Lie algebra and $W$ the corresponding Weyl group, then $(\L a^*)^W = \{0\}$.

\subsection{Ad-invariant Lie quasi-states on simple Lie algebras}
We can now prove the following strengthening of Theorem \ref{MainTheorem1}.
\begin{theorem}\label{MainTheorem1Ext} Let $\gfr$ be a simple real Lie algebra and $\zeta\in {{Q}}_{\rm Ad}(\gfr)$. Then there exists a linear functional $\alpha \in \L{z(k)}^*$ such that for every $X \in \gfr$ with central elliptic part $Y_c \in \L z(\kfr)$ we have
\begin{equation}\label{GWFormula}\zeta(X) = \alpha(Y_c).\end{equation}
In particular,
\[
\dim {{Q}}_{\rm Ad}(\gfr) =\left\{\begin{array}{ll} 1,& \gfr \: \textrm{  Hermitian},\\ 0, & \gfr \: \textrm{  non-Hermitian},\end{array}\right.
\]
and thus ${Q}_{\rm Ad}(\L g) = {\mathcal Q}_{\rm Ad}(\L g)
={\mathcal Q}_{\rm qm}(\L g)$. 
\end{theorem}
\begin{proof} 
Let $\zeta$ be an \rm{Ad-} invariant quasi-state on $\gfr$. By the first two parts of Lemma \ref{AdditiveRJC} we have 
\[\zeta(X) = \zeta(X_c) + \zeta(X_k)+\zeta(X_a)+\zeta(X_n).\] 
Since $\zeta$ is \rm{Ad-} invariant the last part of that lemma yields $\zeta(X_j) =\zeta(Y_j)$ for $j \in \{k,a,n\}$, whence
\begin{equation}\label{JordanApplication}
\zeta(X) = \zeta(Y_c)+\zeta(Y_k)+\zeta(Y_a)+\zeta(Y_n).
\end{equation}
Since $\L{z(k)}$ is abelian the restriction $\alpha := \zeta|_\L{z(k)}$ is a linear functional. It thus remains to show only that $\zeta$ vanishes on $\L a$, $\L n$ and $\L k_{ss}$.

Since $\L a$ is abelian,  $\zeta|_{\L a} \in \L a^*$ is linear, and since $\zeta$ is Ad-invariant, $\zeta|_{\L a}$ is invariant under the adjoint action of the Weyl group $W := N_K(\L a)/Z_K(\L a)$. Now $(\L a^*)^W=\{0\}$ by the remark at the end of the last subsection, and hence $\zeta|_{\L a} = 0$.

Now let $X \in \gfr$ be an arbitrary nilpotent element. By the Jacobson-Morozov theorem (see e.g. \cite[Thm. 10.3]{Knapp}), there exists an embedding $\phi: \L{sl}_2(\R) \hookrightarrow \gfr$ such that $X = \varphi(e)$, where
\[
e = \left(\begin{matrix} 0 & 1\\ 0& 0 \end{matrix}\right).
\]
Now $e$ is conjugate to $2e$ inside $\L{sl}_2(\R)$, since
\[
\left(\begin{matrix} \sqrt 2 & 0\\ 0&  \frac{1}{\sqrt 2} \end{matrix}\right)\cdot\left(\begin{matrix} 0 & 1\\ 0& 0 \end{matrix}\right)\cdot\left(\begin{matrix}  \frac{1}{\sqrt 2} & 0\\ 0&  \sqrt 2 \end{matrix}\right) =  \left(\begin{matrix} 0 & 2\\ 0& 0 \end{matrix}\right),
\]
and hence $X = \phi(e)$ is conjugate to $2X = \phi(2e)$ inside $\L g$. Thus invariance of $\zeta$ implies $\zeta(X) = \zeta(2X) = 2\cdot \zeta(X)$, whence $\zeta(X) = 0$. This shows that $\zeta$ vanishes on all nilpotent elements of $\gfr$, and in particular on $\L n$.

The fact that $\zeta$ vanishes on $\kfr_{ss}$ was already observed in \cite{EP09}. We repeat the argument here for completeness: Since $\kfr_{ss}$ is compact, every Ad-orbit intersects any given maximal toral subalgebra $\L t < \L k_{ss}$, whence $\zeta|_{\kfr_{ss}}$ is determined by $\zeta|_{\L t}$. Since $\L t$ is abelian we have $\zeta|_{\L t}\in \L t^*$ and moreover $\zeta|_{\L t}$ is invariant under the adjoint action of $W_{\kfr} := N_K(\L t)/Z_K(\L t)$. Since $\kfr_{ss}$ is semisimple we have $(\L t^*)^{W_{\kfr}} = \{0\}$ (again by the remark at the end of the last subsection) and thus $\zeta|_{\L t} = 0$ and consequently also $\zeta|_{\kfr_{ss}} = 0$. The theorem follows.
\end{proof}
\begin{remark}
\begin{enumerate}
\item The most substantial ingredient in the proof of Theorem \ref{MainTheorem1Ext} is the vanishing of $\zeta$ along nilpotent elements, which we deduced from the Jacobson-Morozov theorem. If one is willing to assume continuity of $\zeta$ then one can give a more elementary proof of this result which does not invoke the Jacobson-Morozov theorem. Instead one uses the fact that
for every $X \in \L n$ there exists a sequence $(g_n)$ in $G$ such that \[\lim_{n \to \infty}{\rm Ad}(g_n)(X) = 0.\] Assuming continuity of $\zeta$ this is enough to deduce that
\[
0 = \zeta(\lim_{n \to \infty}{\rm Ad}(g_n)(X)) = \lim_{n \to \infty}\zeta({\rm Ad}(g_n)(X)) = \lim_{n \to \infty}\zeta(X) = \zeta(X).
\]
\item In  the case of a simple Hermitian Lie algebra, Theorem \ref{MainTheorem1Ext} gives an explicit formula for all Ad-invariant Lie quasi-states on $\gfr$. Now let $G$ denote the simply-connected Lie group associated with $\gfr$ and $Z(K) \cong \R$ denote the analytic subgroup of $G$ with Lie algebra $\L{z(k)}$. Then for every measurable homomorphism $\chi: Z(K) \to \R$ there is a unique homogeneous quasimorphism $f$ on $G$ subject to the normalisation $f|_{Z(K)} = \chi$, and by Theorem \ref{MainTheorem1Ext} we have
\[
f(\exp(X)) = \chi(\exp(Y_c)).
\]
A global version of this formula was first pointed out in \cite{Surface}.\\

\item The classification of Lie quasi-states on $\gfr$ does not by itself provide a classification of integrated Lie quasi-states on $G$, since there is no general argument which would ensure that two integrated Lie quasi-states with the same directional derivatives coincide. (Such an argument would work e.g. if the Lie quasi-states in question were of class $C^1$, but in that case they are necessarily linear anyway.) It therefore requires additional effort to prove the following global version of Theorem \ref{MainTheorem1Ext}.
\end{enumerate}
\end{remark}
\begin{theorem}\label{GlobalThm} Let $G$ be a connected reductive Lie group and $f: G \to \R$ a conjugation-invariant integrated Lie quasi-state. Then $f$ is a homogeneous quasimorphism.
\end{theorem}

\begin{proof} We first observe that we may assume that $G$ is simply-connected. Indeed, if $G$ is arbitrary and $\widetilde{G}$ is its universal covering group, then every conjugation-invariant integrated Lie quasi-state on $G$ lifts to $\widetilde{G}$, and if this can be shown to be a homogeneous quasimorphism, then it descends to a homogeneous quasimorphism on $G$, see \cite{BSH2}. Now a simply-connected reductive Lie group $G$ is a product of abelian and simple factors. Since every integrated Lie quasi-state on an abelian group is a homomorphism, it suffices to prove the theorem for a general simply-connected simple Lie group $G$.

In this case, we consider the Lie algebra $\gfr$ of $G$ and and fix an Iwasawa decomposition $\gfr =\L k \oplus \L a \oplus \L n$. Let $K$ be a maximal {\rm Ad}-compact subgroup of $G$ with Lie algebra $\L k$ and let $A$ and $N$ be the analytic subgroups associated with $\L a$ and $\L n$. Since $G$ is connected and $K$ is a deformation retract of $G$, also $K$ is connected. Since $K$ is compact-times-abelian and $A$ and $N$ are nilpotent and all three are connected the restricted exponential functions $\L k \to K$,  $\L a \to A$ and $\L n \to N$ are all onto. It thus follows from Theorem \ref{MainTheorem1Ext} that there exists a homogeneous quasimorphism $\phi: G \to \R$ such that $f_0 := f-\phi$ vanishes on $K$, $A$ and $N$. It suffices to show that $f_0 \equiv 0$.

To this end we first observe that since $Z(G) \subset K$ we have $f_0|_{Z(G)}\equiv 0$, and hence $f$ factors through a quasimorphism on $G_0 := {\rm Ad}(G) = G/Z(G)$ which vanishes on the images $K_0$, $A_0$ and $N_0$ of $K$, $A$ and $N$ in $G_0$. The group $G_0$ (unlike $G$) is automatically algebraic, whence admits a multiplicative Jordan decomposition: Every $g \in G$ can be written as a product of elements $g_k$, $g_a$ and $g_n$ which pairwise commute and are conjugate to elements in $K_0$, $A_0$ and $N_0$ respectively. It follows that $f_0(g) = f_0(g_k) + f_0(g_a)+f_0(g_n) = 0$, which finishes the proof.
\end{proof}

\section{Semidirect products and solvable examples}\label{SecAlmostAbelian}
\label{sec:semidirect}

\subsection{Semidirect products and normalized Lie quasi-states}\label{SemidirectGeneral}
The goal of this section is to establish Theorem \ref{MainTheorem2} concerning Ad-invariant Lie quasi-states on certain solvable Lie algebras. Before we turn to the specific setting of that theorem we collect a few general facts about semidirect product algebras of the form $\gfr = V \rtimes \hfr$, where $V\lhd \gfr$ is an abelian ideal and $\hfr < \gfr$ is a complementary subalgebra. We denote by $\rho: \hfr \to \mathfrak{gl}(V)$  the restriction of the adjoint action of $\hfr$ to $V$; then $\rho$ is a representation of $\hfr$ and $\gfr = \gfr(V, \hfr, \rho)$ is uniquely determined by the triple $(V, \hfr, \rho)$. 

Note that if $\alpha \in V^*$ is any linear functional, then the map $(v, X) \mapsto \alpha(v)$ defines a linear functional, hence a Lie quasi-state on $\gfr$. Similarly, if $\zeta \in Q(\hfr)$ is a Lie quasi-state, then so is the map $(v, X)\mapsto \zeta(X)$. This is because the projection onto the second coordinate is a Lie algebra homomorphism and thus $[(v,X), (w, Y)] = 0$ implies $[X,Y] = 0$. In particular, we can consider $V^*$ and $Q(\hfr)$ as subspaces of $Q(\gfr)$.

We define subspaces $Q_0(\gfr) \subset Q(\gfr)$ and $\mathcal Q_0(\gfr) \subset \mathcal Q(\gfr)$ by
\[
Q_{0}(\gfr) = \big\{\zeta \in Q(\gfr)\mid \forall v \in V, X \in \hfr:\;\zeta(v,0) = \zeta(0,X)=0\big\}.
\]
and $\mathcal Q_0(\gfr) := Q_{0}(\gfr) \cap \mathcal Q(\gfr)$. We observe:
\begin{lemma}\label{SemidirectLemma}
Let $\gfr = \gfr(V, \hfr, \rho)$ as above. Then $Q(\gfr)$ can be written as the internal direct sum 
\begin{equation}\label{NormalizedQS}
Q(\gfr) = Q_{0}(\gfr) \oplus V^* \oplus Q(\hfr),\end{equation}
and similarly we have
\[\mathcal Q(\gfr)  = \mathcal Q_{0}(\gfr) \oplus V^* \oplus \mathcal Q(\hfr).\] 
\end{lemma}
\begin{proof} Given $\zeta \in Q(\gfr)$, the restrictions $\zeta|_{V \times \{0\}}$ and $\zeta|_{\{0\} \times \L h}$ are Lie quasi-states. Now $V \times \{0\} \cong V$ is abelian, hence $\alpha(v) :=\zeta(v,0) \in V^*$. Similarly, $\psi(X) := \zeta(0,X)$ defines a Lie quasi-state on $\hfr$. Now let
\[
\zeta_0(v,X) := \zeta(v,X) - \alpha(v) - \psi(X).
\]
Then $\zeta_0 \in Q_{0}(\gfr)$ and $\zeta = \zeta_0+\alpha+\psi$, whence $Q(\gfr) = Q_{0}(\gfr)+V^* +Q(\hfr)$. Since the pairwise intersections of $Q_{0}(\gfr)$, $V^*$ and $Q(\hfr)$ (considered as subspaces of $Q(\gfr)$) are trivial, the sum is direct. Finally, if $\zeta$ is continuous, then so are $\zeta_0, \alpha$ and $\psi$ in the above decomposition.
\end{proof}
In the sequel we will refer to an element $\zeta \in Q_{0}(\gfr)$ as a \emph{normalized Lie quasi-state} on $\gfr$. We record the following property of continuous normalized quasi-states for later use:
\begin{lemma}\label{Sublinearity}
Let $\zeta \in  \mathcal Q_{0}(\gfr)$ be a continuous normalized Lie quasi-state and let $X \in \hfr$. Then the function $\zeta_X: V \to \R$ given by
$\zeta_X(v):= \zeta(v,X)$ is sublinear in the sense that
\[
\lim_{v \to \infty} \frac{\zeta_X(v)}{\|v\|} \to 0
\]
for any norm on $V$.
\end{lemma}
\begin{proof} Since $\zeta$ is continuous, it is uniformly continuous on the product of unit balls. Since it is moreover normalized, we can compute
\[
\lim_{v \to \infty} \frac{\zeta_X(v)}{\|v\|} = \lim_{v \to \infty} \frac{\zeta(v, X)}{\|v\|} = \lim_{v \to \infty} \zeta(v/\|v\|, X/\|v\|) =  \lim_{v \to \infty} \zeta(v/\|v\|, 0)=0.
\]
\end{proof}

\subsection{Classification of Lie quasi-states when $\hfr$ is one-dimensional}

We now specialize to our main case of interest which is given by $\hfr = \R$, i.e. $\gfr = V \rtimes \R$ with $V \lhd \gfr$ abelian. (We will return to the more general situation in the next section.) In this situation we will adopt the following notation: We denote by $\phi \in {\rm End}(V)$ the restriction of the adjoint action of the generator $1 \in \R$ to $V$ so that
\[
[(v,s), (w,t)] = (s\phi(w)-t\phi(v),0).
\]
We then write $\gfr = \gfr_\phi$ and refer to $\phi$ as the \emph{underlying endomorphism} of $\gfr$. Note that if $\phi = 0$ then $\gfr_\phi = V\oplus \R$ is abelian and thus $Q(\gfr) = \gfr^*$. Thus we are going to assume from now on that $\phi \neq 0$.

Given a Lie quasi-state $\zeta \in {Q}(\gfr)$ we refer to the linear functional $\zeta_V \in V^*$ given by 
\[
\zeta_V(v):=\zeta(v,0), \quad \textrm{for all $v \in V$ },
\] 
as the \emph{canonical character} of $\zeta$. The following theorem provides a full classification of Lie quasi-states on $\gfr$. We will denote by $U$ the kernel and by $W$ the image of the endomorphism $\phi$ underlying $\gfr$.
\begin{theorem}\label{SLTheorem}
A function $\zeta: \gfr \to \R$ is a Lie quasi-state of canonical character  $\alpha \in V^*$ if and only if there exists a function $c: W \to \R$  such that
\begin{equation}\label{QSFormulaSemidirect1}
\zeta(v,t) = \left\{\begin{array}{ll}c(\phi(v)/t)\cdot t + \alpha(v) & t \neq 0\\ \alpha(v)& t= 0\end{array}\right.
\end{equation}
The quasi-state $\zeta$ given by \eqref{QSFormulaSemidirect1} is continuous if and only if the function $c$ is continuous and sublinear.
\end{theorem}
\begin{proof} Assume first that $\zeta \in Q_{0}(\gfr)$ is normalized. Then $\zeta(v,0) = 0$ for every $v \in V$ and if $\widetilde{c}: V \to \R$ is given by $\widetilde{c}(v) := \zeta(v,1)$, then for every $v \in V$ and $t \neq 0$ we have
\[
\zeta(v,t) = \widetilde{c}(v/t)\cdot t.
\]
Moreover, since $\zeta$ is normalized we have $\widetilde{c}(0) = \zeta(0,1) = 0$. If $u \in U = \ker(\phi)$, then $(u, 0)$ is central in $\gfr_\phi$, hence
\[
\widetilde{c}(v+u) = \zeta(v+u, 1) = \zeta((v,1)+(u,0)) = \zeta(v,1)+\zeta(u,0) = \zeta(v,1) = \widetilde{c}(v). 
\]
It follows that $\widetilde{c}$ descends to a function on $V/U$. Since the latter space is isomorphic to $W$ via $\phi$, we conclude that there exists a function $c: W \to \R$ with $c(0) = 0$ such that
\begin{equation}\label{QSCandidate}
\zeta(v,t) = \left\{\begin{array}{ll}c(\phi(v)/t)\cdot t & t \neq 0\\ 0 & t= 0.\end{array}\right.
\end{equation}
Conversely, assume that $\zeta: \gfr \to \R$ is given by \eqref{QSCandidate} for some function $c: W \to \R$ with $c(0) = 0$. Then, by definition, $\zeta(v, 0) = \zeta(0,t) = 0$ for all $v\in V$ and $t\in \R$. Now assume that $(u, s)$ and $(v,t)$ commute, i.e. $s\phi(v) = t\phi(u)$. We claim that $\zeta(u,s) + \zeta(v,t) = \zeta(u+v, s+t)$. There are several cases:
\begin{enumerate}[\textsc{Case} 1.]
\item If $0 \not \in \{s,t, s+t\}$ then we have 
\[
\frac{\phi(u)}{s} = \frac{\phi(v)}{t} = \frac{\phi(u+v)}{s+t},
\]
and we deduce that
\begin{eqnarray*}
\zeta(u,s)+\zeta(v,t) &=& c(\phi(u)/s) \cdot s + c(\phi(v)/t)\cdot t = c(\phi(u+v)/(s+t)) \cdot s +  c(\phi(u+v)/(s+t)) \cdot t\\
&=& c(\phi(u+v)/(s+t)) \cdot (s+t) = \zeta(u+v, s+t).
\end{eqnarray*}
\item If $s = 0$, then the condition $s\phi(v) = t\phi(u)$ implies either $t=0$ or $\phi(u) = 0$. In the former case we have
$\zeta(u,s)+\zeta(v,t) = 0+0 = 0 = \zeta(u+v, s+t)$, and in the latter case we have
\[
\zeta(u,s)+\zeta(v,t) = 0+ c(\phi(v)/t)\cdot t  = c((0+\phi(v))/(0+t)) \cdot (0+t) = c((\phi(u)+\phi(v))/(s+t)) \cdot (s+t) =\zeta(u+v, s+t).
\]
\item If $t = 0$ then we can argue as in \textsc{Case 2}, since the roles of $s$ and $t$ are symmetric.

\item It remains to deal with the case $s=-t \neq 0$. In this case,
\[
\frac{\phi(u)}{s} = \frac{\phi(v)}{t},
\]
and thus
\begin{eqnarray*}
\zeta(u,s)+\zeta(v,t) &=& c(\phi(u)/s) \cdot s + c(\phi(v)/t)\cdot t = c(\phi(u)/s) \cdot s + c(\phi(u)/s)\cdot t \\
&=&  c(\phi(u)/s) \cdot (s+t) = 0 = \zeta(u+v, 0) = \zeta(u+v, s+t).
\end{eqnarray*}
\end{enumerate}
It follows that $\zeta$ is a normalized Lie quasi-state. We have thus shown that the normalized Lie quasi-states on $\gfr$ are exactly the functions of the form \eqref{QSCandidate}, where $c: W \to \R$ is an arbitrary function with $c(0) = 0$. 

This finishes the classification of normalized Lie quasi-states on $\gfr$. As for general Lie quasi-states, by Lemma \ref{SemidirectLemma} these are of the form $\zeta(v,t) =  \zeta_0(v,t)+\alpha(v)+ c_0\cdot t$, where $\zeta_0 \in Q_0(\gfr)$ is normalized, $\alpha \in V^*$ denotes the central character of $\zeta$ and $c_0 \in \R$ is given by $c_0 = \zeta(0,1)$. Since $\zeta_0$ is of the form \eqref{QSCandidate} for some $c: W \to \R$ with $c(0) = 0$, it follows that $\zeta$ is of the form described in the statement of the theorem.

As for continuity, if the quasi-state $\zeta$ given by \eqref{QSFormulaSemidirect1} is continuous, then in particular the map $v \mapsto c(\phi(v)) = \zeta(v, 1)$ is continuous, whence $c$ is continuous. Moreover, continuity of $\zeta$ at $t = 0$ implies sublinearity of $c$ by Lemma \ref{Sublinearity}. Conversely, if $c$ is continuous, then $\zeta$ is obviously continuous on $\gfr \setminus(V \times \{0\})$, and if $c$ is moreover sublinear, then $\zeta$ is continuous on all of $\gfr$.
\end{proof}
In the sequel we write $\zeta_{\alpha, c}$ for the Lie quasi-state on $\gfr$ given by equation \eqref{QSFormulaSemidirect1}. If we denote by $C_{sl}(W) < C(W)$ the space of sub-linear continuous functions on $W$ then we have:
\begin{corollary} If $\gfr = \gfr_\phi$ for some $\phi \neq 0$ and $W = {\rm im}(\phi)$ then the map
\[
V^*\oplus C_{sl}(W) \to \mathcal Q(\gfr), \quad (\alpha, c)\mapsto \zeta_{\alpha,c}.
\]
is an isomorphism and thus ${\mathcal Q}(\gfr)$ is infinite-dimensional.
\end{corollary}

\subsection{The Ad-invariant case} 
We keep the notation of the previous subsection, in particular $V$ is an abelian Lie algebra, $\phi \in {\rm End}(V) \setminus\{0\}$ and
$\L g = \L g_{\phi} = V \rtimes \R$ with ${\rm ad}(1)|_V = \phi$. Moreover, we abbreviate $W := {\rm im}(\phi)$ and given $\alpha \in V^*$ and $c \in C_{sl}(W)$ we denote by $\zeta_{\alpha, c}$ the Lie quasi-state on $\gfr$ given by \eqref{QSFormulaSemidirect1}. The following theorem classifies those quasi-states $\zeta_{\alpha,c}$ which are invariant under the action of the adjoint group of $\gfr$.
\begin{theorem}\label{SolvableAdInv}
Let $\gfr = \gfr_\phi$ as above. Given $\alpha \in V^*$ and $c \in C_{sl}(W)$, the Lie quasi-state $\zeta_{\alpha, c}$ is ${\rm Ad}$-invariant if and only if the following hold:
\begin{enumerate}[(i)]
\item $\im \phi \subset \ker \alpha$.
\item $c: W \to \R$ descends to a function $c : W/\im \phi^2 \ra \bR$.
\end{enumerate}
\end{theorem}
\begin{proof} The function $\zeta_{\alpha,c}$ is ${\rm Ad}$-invariant if and only if for all $(w,s), (v,t) \in \gfr$ we have
\begin{equation}\label{AdInvNew}
\zeta_{\alpha, c}(\exp(\ad(w,s))(v,t)) = \zeta_{\alpha, c}(v,t)
\end{equation}
In fact, by continuity of $\zeta_{\alpha, c}$ we may assume that $s \neq 0 \neq t$. Note that  
\[
\ad(w,s)(v,t) = \left(\begin{matrix} s\phi & -\phi(w)\\0&0\end{matrix}\right) \cdot \left(\begin{matrix}v\\ t\end{matrix}\right),
\]
and hence for all $k >0$,
\[
\ad(w,s)^k(v,t) = \left(\begin{matrix} s\phi & -\phi(w)\\0&0\end{matrix}\right)^k \cdot \left(\begin{matrix}v\\ t\end{matrix}\right) =  \left(\begin{matrix} s^k\phi^k & -s^{k-1}\phi^k(w)\\0&0\end{matrix}\right) \cdot \left(\begin{matrix}v\\ t\end{matrix}\right) = 
\left(\begin{matrix}s^k\phi^k(v)-ts^{k-1}\phi^k(w)\\0 \end{matrix}\right).
\]
We deduce that 
\begin{eqnarray*}
\exp(\ad(w,s))(v,t) &=& \left(\begin{matrix} 1+\sum_{k=1}^\infty \frac{1}{k!} s^k\phi^k & -\sum_{k=1}^\infty \frac{1}{k!} s^{k-1}\phi^k(w)\\0&1\end{matrix}\right) \cdot \left(\begin{matrix}v\\ t\end{matrix}\right) \\
&=& \left(\begin{matrix} \exp(s\phi) & \frac{1}{s}(\exp(s\phi)w - w)\\0&1\end{matrix}\right) \cdot \left(\begin{matrix}v\\ t\end{matrix}\right) \\
&=& \Big(\exp(s\phi)v - \frac{t}{s}\big(\exp(s\phi)w-w\big),t\Big), 
\end{eqnarray*}
and thus the condition \eqref{AdInvNew} amounts to
\begin{eqnarray}\label{AdInvNew2}
&&c\Big(\frac{\exp(s\phi)\phi(v)}{t} - \frac{1}{s}\Big(\exp(s\phi)\phi(w)-\phi(w)\Big)\Big)\cdot t + \alpha(\exp(s\phi)v) - t \cdot \alpha\Big(\frac{1}{s}\Big(\exp(s\phi)w-w\Big)\Big) \\
&=& \nonumber c\Big(\frac{\phi(v)}{t}\Big)\cdot t + \alpha(v).
\end{eqnarray}
We now show that \eqref{AdInvNew2} implies (i) and (ii) above. Specialize \eqref{AdInvNew2} to $t:=1$ and then let $s \to 0$. Since
\[
\lim_{s \ra 0} \frac{\exp(s\phi)u - u}{s} = \phi(u) \quad \textrm{for all $u \in V$}, 
\]
we obtain
\begin{equation}\label{AdInvNew3}
c\big(\phi(v) - \phi^2(w)\big) - \alpha(\phi(w)) = c(\phi(v)).
\end{equation}
Replacing $w$ by $rw$ for some $r > 0$ and dividing by $r$ we obtain
\[
\frac 1 r \cdot c\big(\phi(v) - r\cdot \phi^2(w)\big) - \alpha(\phi(w)) = \frac{c(\phi(v))}{r}.
\]
If $\phi^2(w) \neq 0$, then $\frac 1 r \|\phi(v) - r\cdot \phi^2(w)\| \to \|\phi^2(w)\|$ and thus sublinearity of $c$ implies that the first term converges to $0$ as $r \to \infty$. This in term implies $\alpha(\phi(w)) = 0$. If, on the other hand, $\phi^2(w) = 0$, then \eqref{AdInvNew3} simplifies to $\alpha(\phi(w)) = \alpha(v)$ for all $v \in V$ and choosing $v =0$ yields $\alpha(\phi(w)) = 0$ also in this case. We have thus established (i). This in turn allows us to rewrite \eqref{AdInvNew3}  as
\begin{equation}\label{AdInvNew4}
c\big(\phi(v) - \phi^2(w)\big)= c(\phi(v)).
\end{equation}
which is (ii). Hence we have seen that Ad-invariance implies both (i) and (ii).

Towards the converse implication, observe that
\[
\exp(s\phi)\phi(v) \in \phi(v) + \im \phi^2
\]
and similarly $\exp(s\phi)\phi(w) \in \phi(w) + \im \phi^2$. It follows that
\[\frac{\exp(s\phi)\phi(v)}{t} - \frac{1}{s}\Big(\exp(s\phi)\phi(w)-\phi(w)\Big) \in \frac{\phi(v)}{t} + \im \phi^2
\]
and thus
\[
c\left(\frac{\exp(s\phi)\phi(v)}{t} - \frac{1}{s}\Big(\exp(s\phi)\phi(w)-\phi(w)\Big)\right)= c\left(\frac{\phi(v)}{t} \right).
\]
for any $c$ satisfying (ii). Thus if (ii) holds, then the Ad-invariance condition \eqref{AdInvNew2} simplifies into
\[\alpha\left(\exp(s\phi)(v) -\frac{t}{s}\Big(\exp(s\phi)w-w\Big)\right)=\alpha(v).\]
Since
\[
\exp(s\phi)(v) -\frac{t}{s}\Big(\exp(s\phi)w-w\Big) \in v + \im \phi,
\]
this follows from (i), and hence (i) and (ii) together imply Ad-invariance.
\end{proof}
The theorem allows us to determine whether the algebra $\gfr_\phi$ is Lie quasi-state rigid. We can summarize the result as follows:
\begin{corollary}\label{SplittingNew} The following are equivalent for the Lie algebra $\gfr = \gfr_\phi$:
\begin{enumerate}[(i)]
\item $\gfr_\phi$ is not Lie quasi-state rigid.
\item $\mathcal Q_{Ad}(\gfr)$ is infinite-dimensional (and, in fact, of uncountable dimension).
\item $\im \phi^2 \subsetneq \im \phi$.
\item There exists a $\phi$-invariant splitting $V = V_0 \oplus V_1$ with $n := \dim V_0 \geq 2$ such that in some suitable basis of $V_0$ the restriction $\phi|_{V_0}$ is represented by the matrix
\begin{equation}\label{phin}
\phi_n = \left( 
\begin{matrix}
0&1& &&&&\\
&0&1&&&\\
&&\ddots&\ddots&&\\
&&&0&1\\
&&&&0
\end{matrix}
\right)
\end{equation}
\item $\gfr_\phi$ surjects onto the three-dimensional Heisenberg algebra $\L h_3$.
\end{enumerate}
\end{corollary}
Note that the equivalence (i)$\Leftrightarrow$(v) is precisely Theorem \ref{MainTheorem2} from the introduction. The proof of Corollary \ref{SplittingNew} will make use of some basic functoriality properties of Lie quasi-states which we list in the following lemma:
\begin{lemma}\label{QuotientMaps} Let $p: \L g \to \L q$ be a surjective Lie algebra homomorphism and let $\zeta \in \mathcal Q(\L q)$ be a quasi-state.
\begin{enumerate}[(i)]
\item The functions $p^*\zeta: \gfr \to \R$ given by
\[
p^*\zeta(X) := \zeta(p(X))
\]
is a Lie quasi-state on $\L g$. 
\item If $\zeta$ is Ad-invariant, then so is $p^*\zeta$.
\item If $\zeta  \in Q_{qm}(\L q)$, then $p^*\zeta \in Q_{qm}(\gfr)$
\end{enumerate}
\end{lemma}
\begin{proof} (i) If $X, Y \in \gfr$ with $[X, Y] = 0$, then $[p(X), p(Y)] = p([X, Y]) = 0$ and thus
\[
p^*\zeta(X+Y) = \zeta(p(X) + p(Y)) = \zeta(p(X)) + \zeta(p(Y)) = p^*\zeta(X) + p^*\zeta(Y),
\]
and $\R$-linearity is obvious.

(ii) Assume that $\zeta$ is Ad-invariant. Then for all $X, Y \in \gfr$ we have
\begin{eqnarray*}
p(\Ad(\exp(X))(Y)) &=& p(\exp(\ad(X))(Y)) \\
&=&p\left( \sum_{k=0}^\infty \frac{1}{k!} \ad(X)^k(Y)\right)\\
&=&\sum_{k=0}^\infty \frac{1}{k!} \ad(p(X))^k(p(Y))\\
&=&\exp(\ad(p(X))(p(Y))\\
&=&\Ad(\exp(p(X)))(p(Y))
\end{eqnarray*}
and thus
\[
p^*\zeta(\Ad(\exp(X))(Y)) = \zeta(p(\Ad(\exp(X)(Y))  ) =\zeta(\Ad(\exp(p(X)))(p(Y))) = \zeta(p(Y)) = p^*\zeta(Y),
\]
which implies Ad-invariance of $p^*\zeta$.

(iii) Denote by $Q$ and $G$ the $1$-connected Lie groups associated with $\L q$ and $\L g$ respectively and denote by $\widehat{p}: G \to Q$ the lift of $p$. If $\zeta  \in Q_{qm}(\L q)$ is the directional derivative of a homogeneous quasimorphism $f: Q \to \R$, then 
\[
p^*\zeta = p^*(f \circ \exp) = f\circ \exp \circ p = f \circ \widehat{p} \circ \exp = \widehat{p}^*f \circ \exp, 
\]
whence $p^*\zeta$ is the directional derivative of $\widehat{p}^*f$.
\end{proof}
\begin{proof}[Proof of Corollary \ref{SplittingNew}] We first show the equivalence of the first three conditions:

(i) $\Rightarrow$ (iii): Assume $\im(\phi^2) = \im(\phi)$ and let $\zeta_{\alpha, c} \in \mathcal Q_{Ad}(\gfr)$. Theorem \ref{SolvableAdInv} then implies that $c$ is constant, whence $\zeta_{\alpha, c}$ is linear. This show that $\gfr$ is Lie quasi-state rigid.

(iii) $\Rightarrow$ (ii): If (iii) holds, then $\dim(\im \phi/\im \phi^2)>0$, whence $C_{sl}(\im \phi/\im \phi^2)$ is of uncountable dimension, and the map $c\mapsto \zeta_{0, c}$ defines a linear embedding of $C_{sl}(\im \phi/\im \phi^2)$ into $\mathcal Q_{\rm Ad}(\gfr_\phi)$.

(ii) $\Rightarrow$ (i): We already observed in the introduction that if $\gfr$ is quasi-state rigid, then $\dim \mathcal Q_{\rm Ad}(\gfr_\phi) < \infty$.

On the other hand, the equivalence (iii) $\Leftrightarrow$ (iv) is an immediate consequence of the real Jordan decomposition of $\phi$. Namely, $\im \phi^2 \subsetneq \im \phi$ if and only if $\phi$ contains a nilpotent Jordan block of size $\geq 2$. We have thus established equivalence of the first four properties. 

Before we turn to condition (v) we observe that if $\phi_2 \in {\rm End}(\R^2)$ is given by
\[
\phi_2 = \left(\begin{matrix} 0 & 1\\ 0 &0 \end{matrix}\right),
\]
then $\gfr_{\phi_2}$ is precisely the three-dimensional Heisenberg algebra $\L h_3$. It thus follows from the implication (iv)$\Rightarrow$(ii), that $\mathcal Q_{Ad}(\L h_3)$ is infinite-dimensional. Now we can establish the remaining implications:

(v) $\Rightarrow$(ii): Assume $p: \gfr \to \L h_3$ is a surjection. By Lemma \ref{QuotientMaps}, $p$ induces a map $p^*:  \mathcal Q_{Ad}(\L h_3) \to  \mathcal Q_{Ad}(\L g)$, and since $p$ is surjective, this map is injective. It follows that $\dim \mathcal Q_{Ad}(\L g) \geq \dim \mathcal Q_{Ad}(\L h_3) = \infty$.

(iv) $\Rightarrow$ (v): Assume that $\gfr_\phi$ satisfies (iv) and let $V_0$, $V_1$ be as described in (iv). Then $V_1 \lhd \gfr_\phi$ is an ideal, and $\gfr_\phi/V_1 \cong \gfr_{\phi_n}$, where $\phi_n \in {\rm End}(\R^n)$ is given as in (iv). By definition the Lie algebra $\gfr_{\phi_n}$ has a basis of the form $(X, Y_1, \dots, Y_n)$ with bracket relations
\[
[X, Y_1] = Y_2,\; \dots, \; [X, Y_{n-1}] = Y_n,\; [X, Y_n] = 0,\; [Y_i, Y_j] = 0\quad(i,j = 1, \dots, n).
\]
In this basis the center $\L z(\gfr_{\phi_n})$ of $\gfr_{\phi_n}$ is given by $\L z(\gfr_{\phi_n}) = \R \cdot Y_n$, and we have $\gfr_{\phi_n}/\L z(\gfr_{\phi_n})  = \gfr_{\phi_{n-1}}$. Thus if $\gfr_{\phi}$ satisfies (iv), then we have a chain of surjections
\[
\gfr_{\phi} \to \gfr_{\phi_n} \to \gfr_{\phi_{n-1}} \to \dots \to \gfr_{\phi_3} \to \gfr_{\phi_2}  = \L h_3,
\]
and thus $\gfr_\phi$ surjects onto $\L h_3$. This finishes the proof.
\end{proof}

\begin{remark} 
Recall that if the space of homogeneous quasimorphisms modulo homomorphisms on some group $G$ is infinite-dimensional, then its dimension is at least the cardinality of the continuum. Indeed, this follows from the fact that the second bounded cohomology of a group is a Banach space with respect to the Gromov norm (and thus any Hamel basis has either finite or uncountable cardinality). This should be compared to the equivalence (i)$\Leftrightarrow$(ii) above, which allows the even stronger conclusion that the space $\mathcal Q_{\rm Ad}(\gfr_\phi)/{\rm Hom}(\gfr_\phi, \R)$ is of uncountable dimension as soon as it is non-trivial. We do not know whether for an arbitrary Lie algebra $\gfr$ the dimension of $\mathcal Q_{\rm Ad}(\gfr)/{\rm Hom}(\gfr, \R)$ can be countable-dimensional. For a solvable Lie algebra $\gfr$ we do not even know whether it can be finite-dimensional without being $0$.
\end{remark}
\subsection{Low-dimensional examples}\label{SolvableExample}
Let us summarize what the results obtained so far imply for low-dimensional Lie algebras. Every one-dimensional Lie algebra is abelian and thus rigid. Every non-abelian two-dimensional Lie algebra is isomorphic to the Lie algebra of the $(ax+b)$-group. This Lie algebra can be realized as $\gfr_{\phi}$, where $\phi={\rm Id}_{\R}$, and thus is rigid by the criterion from Corollary \ref{SplittingNew}.(iii).

Now assume that $\gfr$ is a three-dimensional real Lie algebra. If $\gfr$ contains a simple subalgebra, then it is actually simple (namely $\gfr$ is one of the two three-dimensional simple Lie algebras  $\L{sl}_2(\R)$ or $\L{su}(2)$), and thus rigid. Otherwise $\gfr$ is a solvable $3$-dimensional Lie algebra. We now recall the classification of such Lie algebras.

In order to determine a three-dimensional Lie algebra $\gfr$ it suffices to compute the commutators $[X_3, X_1]$, $[X_3, X_2]$ and $[X_1, X_2]$ for some basis $(X_1, X_2, X_3)$ of $\gfr$. The following table thus determines four families of three-dimensional Lie algebras:
\begin{center}
\begin{tabular}{ l | c | c | c }
 Name of the Lie algebra in \cite{deGraaf} & $[X_3, X_1]$ & $[X_3, X_2]$&$[X_1, X_2]$ \\
 \hline
  $L^1$ & $0$ & $0$ & $0$ \\
  $L^2$ & $X_1$ & $X_2$ & $0$ \\
  $L^3_a$ ($a \in \R$)& $X_2$ & $aX_1+X_2$ & $0$ \\
  $L^4_a$ ($a \in \R$) & $X_2$ & $aX_1$ & $0$ 
\end{tabular}
\end{center}
According to \cite{deGraaf}, every three-dimensional solvable Lie algebra is isomorphic to a Lie algebra in one of these families. Moreover, Lie algebras from different families are never isomorphic, Lie algebras of the form $L^3_a$ are pairwise non-isomorphic and $L^4_a = L^4_b$ if and only if $a = \lambda \cdot b$ for some $\lambda >0$. The Lie algebra $L^1$ is abelian, and in particular rigid. If $\gfr \in \{L^2, L^3_a, L^4_a\}$, then the subspace $V = {\rm span}(X_1, X_2)$ is a $2$-dimensional abelian ideal of $\gfr$, for which the short exact sequence
\[
0 \to V \to \gfr \to \R \to 0
\]
splits. It follows that $\gfr$ is of the form $\gfr = \gfr_{\phi}$, where $\phi$ is given by
\[
\phi^{(2)} = \left(\begin{matrix} 1&0\\0&1\end{matrix}\right), \quad \phi^{(3,a)} = \left(\begin{matrix} 0&1\\a&1\end{matrix}\right),  \quad \phi^{(4,a)} = \left(\begin{matrix} 0&1\\a&0\end{matrix}\right),
\]
according to whether $\gfr$ is equal to $L^2$, $L^3_a$ or $L^4_a$. Note that the algebra $L^4_0$ is isomorphic to the three-dimensional Heisenberg algebra $\L h_3$. It it immediate from the criterion in Corollary \ref{SplittingNew}.(v) that $\L h_3$ is the only non-rigid three-dimensional Lie algebra of the form $\gfr_\phi$. Since, by the above classification, every solvable Lie algebra of dimension $3$ is of the form $\gfr_\phi$ for some $\phi$, it follows that $\L h_3$ is also the unique non-rigid solvable three-dimensional Lie algebra. Combining this with our previous observations we have proved:
\begin{theorem}\label{ThmLow}
A Lie algebra of dimension $\leq 3$ is rigid if and only if it is not isomorphic to the three-dimensional Heisenberg algebra $\L h_3$.
\end{theorem}
For example, the Lie algebras of the three-dimensional SOL group and of the $1$-dimensional unitary motion group $\C \rtimes U(1)$ are rigid.

\subsection{The three-dimensional Heisenberg algebra}
Since the three-dimensional Heisenberg algebra $\L h_3$ plays such a prominent role in our classification, let us describe the space of Ad-invariant Lie quasi-states on $\L h_3$ explicitly. Here we will think of $\L h_3$ as the Lie algebra of strictly upper triangular $3 \times 3$-matrices.
\begin{corollary}\label{CorH3}  Let $a \in \R$ and let $c$ be a sublinear continuous function on $\R$. Then the function 
\[
\zeta_{\alpha, c}\left(\left(\begin{matrix}0&x&z\\0&0&y\\0&0&0\end{matrix} \right)\right) = c(y/x) \cdot x + a \cdot y \quad(x \in \R^\times, y, z \in \R)
\]
extends continuously to an Ad-invariant quasi state $\zeta_{\alpha, c}: \L h_3 \to \R$. Moreover, every continuous Ad-invariant quasi-state on $\L h_3$ is of this form.
\end{corollary}
\begin{proof} Let $(X, Y, Z)$ be the basis of $\L h_3$ given by the elementary matrices  $X = E_{1,2}$, $Y = E_{2,3}$ and $Z= E_{1, 3}$ so that
\[
[X, Y] = Z, \; [X, Z] = 0, \; [Y, Z] = 0.
\]
Then $V  := {\rm span}(Y, Z)\lhd \gfr$ is a $2$-dimensional ideal and $\gfr = \gfr_\phi$, where $\phi \in {\rm End}(V)$ is given by $\phi(Y) = Z$, $\phi(Z) = 0$. In particular, $\ker \phi = \im \phi = \R \cdot Z$ and $\im \phi^2 = 0$. It then follows from Theorem \ref{SolvableAdInv} that the Ad-invariant Lie quasi-states on $\L h_3$ are of the form $\zeta=\zeta_{\alpha, c}$, where $\alpha$ is a linear functional on $V$ which vanishes on $\R\cdot Z$ and $c$ is a sublinear continuous function on $\R \cdot Z$. Let $a \in \R$ such that $\alpha(yY+zZ) = a \cdot y$ and identify $c$ with a sublinear continuous function on $\R$ via the isomorphism $\lambda \mapsto \lambda \cdot Z$. Then we obtain for all $x \neq 0$, 
\[
\zeta_{\alpha, c}(xX+yY+zZ) = c(\phi(yY+zZ)/x) \cdot x + \alpha(yY+zZ) = c(y/x) \cdot x + a \cdot y,
\]
which finishes the proof.
\end{proof}

\section{Unitary motion algebras}
\label{sec:unitary}

In this section we shall consider the family of unitary motion algebras $\L g_n := \C^n \rtimes \L{u}(n)$. We have already seen in Subsection \ref{SolvableExample} that the Lie algebra $\L g_1$ is rigid and that $\mathcal Q(\gfr_1)$ is infinite-dimensional. For the rest of this section we will thus assume $n \geq 2$, unless otherwise mentioned.

\subsection{Normalized Lie quasi-states and generalized frame function}
Recall from Subsection \ref{SemidirectGeneral} that a Lie quasi-state $\zeta: \gfr_n \to \R$ is called \emph{normalized} provided
\[
\zeta(v,0) = \zeta(0, X) = 0
\]
for all $v \in \C^n$ and $X \in \L u(n)$. We are going to show that for $n \geq 2$ every normalized continuous Lie quasi-state is actually trivial, thereby reducing the study of continuous Lie quasi-states on $\L g_n$ to the separate study of Lie quasi-states on $\C^n$ and $\L u(n)$. Recall that in \cite{Gl57}, Gleason established linearity of Lie quasi-states on $\L u(n)$, $n \geq 3$, by relating them to so-called frame functions and showing that such frame functions are necessarily of a special form. We will follow a similar strategy and relate normalized Lie quasi-states on $\gfr_n$ to generalized frame functions in the sense of the following definition:
\begin{definition} A function $f: \C^n \to \R$ is called a \emph{generalized frame function} if it satisfies $f(0) = 0$ and for every pair $(u,v)$ of orthogonal vectors in $\C^n \setminus\{0\}$ we have
\[
f(u+v) = f(u)+f(v).
\]
\end{definition}
Obviously, every homomorphism $f: \C^n \to \R$ is a generalized frame function, in particular there exist generalized frame function of linear growth in arbitrary dimensions. However, as we establish in Theorem \ref{FFTrivial} in the appendix, there are no frame functions of \emph{sublinear} growth for any $n \geq 2$. Here a function $f: \C^n \to \R$ is called sublinear provided 
\[
\lim_{v \to \infty} \frac{f(v)}{\|v\|} = 0.
\]
The following propositions links generalized frame functions to normalized Lie quasi-states on $\L g_n$. 

\begin{proposition}\label{FrameSL} Let $n \geq 2$ and let $\zeta: \gfr_n \to \R$ be a normalized Lie quasi-state. Then the following hold: 
\begin{enumerate}
\item The function $f_\zeta: \C^n \to \R$ given by 
\[
f_\zeta(w) := \zeta(-iw, i\cdot{\bf 1})
\]
is a generalized frame function. 
\item If $\zeta$ is continuous, then $f_\zeta$ is continuous and sublinear.
\item $\zeta$ is  uniquely determined by $f_\zeta$. In particular, $f_\zeta = 0$ implies $\zeta = 0$.
\end{enumerate}
\end{proposition}
\begin{proof} For the proof we introduce the following notation: Firstly, we
denote by $\langle \cdot, \cdot \rangle$ an inner product on $\C^n$ with respect to which elements of $\L u(n)$ are skew-Hermitian. We use the convention that $\langle \cdot, \cdot \rangle$ is linear in the first argument and anti-linear in the second argument. Given a vector $v \in \C^n\setminus\{0\}$ we define $P_v \in \L u(n)$ by
\[
P_v(w) = i\frac{\langle w,v \rangle}{\langle v, v\rangle} v,
\]
so that $-iP_v$ is the orthogonal projection onto the line spanned by $v$. Note that $P_v = P_w$ if and only if $v =\lambda w$ for some $\lambda \in \C^\times$ and that for every orthogonal basis $(v_1, \dots, v_n)$ we have 
\begin{equation}\label{ProjectorSum}
\sum_{j=1}^nP_{v_j} = i\cdot {\bf 1}.\end{equation}
We claim that for all $w \in \C^n \setminus\{0\}$ we have
\begin{equation}\label{EqProjector}
f_\zeta(w) = \zeta(-iw, P_w).
\end{equation}
Observe first that if $Q \in \L u(n)$ satisfies $Qv  = 0$ for some $v \in \C^n$, then for all $P \in \L u(n)$ that commutes with $Q$ we have $[(v, P), (0, Q)] = 0$ and hence
\[ 
\zeta(v, P+Q) = \zeta(v, P) + \zeta(0, Q) = \zeta(v,P).
\]
In particular, if $u \in \C^n \setminus\{0\}$ we can apply this to $u:= -iv$, $P:= P_v$ and $Q := i\cdot {\bf 1}-P_v$ to obtain
\[
f_\zeta(v) = \zeta(-iv,  i\cdot{\bf 1}) = \zeta(-iv, P_v+Q) = \zeta(-iv, P_v),
\]
which establishes \eqref{EqProjector}. We now use this to establish (i)-(iii):

(i) If $u,v \in \C^n \setminus\{0\}$ are orthogonal and $Q := i \bf{\bf 1} - P_u - P_v$, then the elements $(-iu, P_u)$, $(-iv, P_v)$ and $(0, Q)$ pairwise commute, and we deduce with \eqref{EqProjector} that
\[
f_\zeta(u+v) = \zeta(-iu-iv, P_u+P_v+Q) = \zeta(-iu, P_u)+\zeta(-iv, P_v)+\zeta(0,Q) =  \zeta(-iu, P_u)+\zeta(-iv, P_v) =f_\zeta(u) + f_\zeta(v).
\]
Since $\zeta$ is normalized we also have $f_\zeta(0) = 0$, and thus $f_\zeta$ is indeed a generalized frame function.

(ii) Continuity of $f_\zeta$ is immediate from continuity of $\zeta$, and sublinearity follows from Lemma \ref{Sublinearity}.

(iii) Let $w \in \C^n$ and let $X \in \L u(n)$. We are going to express $\zeta(w,X)$ in terms of the function $f:=f_\zeta$. For this let $(v_1, \dots, v_n)$ be an orthonormal basis of eigenvectors of $X$ with corresponding (purely imaginary) eigenvalues $(i\lambda_1, \dots, i\lambda_n)$. Note that $X \in \L u(n)$ implies that $\lambda_j \in \R$, so we can order them by decreasing absolute value and assume that $|\lambda_1| \geq \dots \geq |\lambda_l| > \lambda_{l+1} = \dots = \lambda_n = 0$, 
for some index $l = l(X)$. If we wish to further emphasize the dependence on $X$ we write $\lambda_j(X)$ and $v_j(X)$ instead of $\lambda_j$ and $v$. 

Now let $w \in \C^n$ and set $w_j := P_{v_j}(w)$. We observe that by \eqref{ProjectorSum},
\[
X.w = X\left(-i\cdot \sum_{j=1}^nP_{v_j}(w)\right) = \sum_{j=1}^n \lambda_j P_{v_j}(w).
\]
and
\[
w=-i \sum_{j=1}^n w_j.
\]
We thus obtain
\[
(w, X) = \sum_{j=1}^n (-iw_j, \lambda_jP_{v_j}) 
\]
Since for $j \neq k$ we have $\lambda_jP_{v_j}(-iw_k) = 0 =\lambda_kP_{v_k}(-iw_j)$ the summands on the right hand side commute, and thus
\begin{equation}\label{EqZetawX}
\zeta(w, X) = \sum_{j=1}^n \zeta(-iw_j, \lambda_jP_{v_j})=  \sum_{j=1}^l \zeta(-iw_j, \lambda_jP_{v_j}) 
\end{equation}
Now assume that $j \leq l$. If $w_j \neq 0$,  then $w_j/\lambda_j$ is proportional to $v_j$, whence $P_{w_j/\lambda_j} = P_{v_j}$. In this case we thus have
\begin{eqnarray*}
\zeta(-iw_j, \lambda_jP_{v_j}) &=& \lambda_j \cdot \zeta(-iw_j/\lambda_j, P_{v_j})\\ 
&=& \lambda_j \cdot \zeta(-iw_j/\lambda_j, P_{w_j/\lambda_j})\\
&=& \lambda_j \cdot f(w_j/\lambda_j)\\
&=& \lambda_j \cdot f(P_{v_j}(w)/\lambda_j), 
\end{eqnarray*}
i.e.
\begin{equation}\label{wjEq}
\zeta(-iw_j, \lambda_jP_{v_j})=\lambda_j \cdot f(P_{v_j}(w)/\lambda_j).
\end{equation}
If $w_j = 0$ then \eqref{wjEq} still holds, since both sides of the equation are $0$ by our normalisation. Plugging \eqref{wjEq} into \eqref{EqZetawX} we obtain
\begin{equation*}\label{QSfromFF}
\zeta(w,X) = \sum_{j=1}^{l(X)} \lambda_j(X) \cdot f_\zeta(P_{v_j(X)}(w)/\lambda_j(X)).
\end{equation*}
This shows in particular that $\zeta$ is uniquely determined by $f_\zeta$.
\end{proof}
Combining this with the aforementioned Theorem \ref{FFTrivial} we deduce:
\begin{corollary}\label{NCLQSTrivial} Let $n \geq 2$. Then every normalized continuous Lie quasi-state $\zeta: \gfr_n \to \R$ is trivial.
\end{corollary}
\begin{proof} If $\zeta$ is any continuous quasi-state on $\gfr_n$, then $f_\zeta$ is a sublinear generalized frame function by Proposition \ref{FrameSL}. We thus deduce from Theorem \ref{FFTrivial} that $f_\zeta \equiv 0$, which in turn implies $\zeta \equiv 0$ by invoking Proposition \ref{FrameSL} again.
\end{proof}

\subsection{The case $n \geq 3$}
Recall from Lemma \ref{SemidirectLemma} that the space of continuous Lie quasi-states on $\L g_n$ is given by
\begin{equation*}\label{UMADec}
\mathcal Q(\L g_n) = \mathcal Q_{0}(\L g_n) \oplus (\C^n)^* \oplus \mathcal Q(\L u(n)), \end{equation*}
where $\mathcal Q_{0}(\L g_n)$ denotes the subspace of normalised continuous Lie quasi-states. We have just established in Corollary \ref{NCLQSTrivial} that this space is trivial, whence
\begin{equation}\label{NoMixedQS}
\mathcal Q(\L g_n) =(\C^n)^* \oplus \mathcal Q(\L u(n)),
\end{equation}
i.e. every continuous quasi-state on $\L g_n$ decomposes into a sum of continuous quasi-states on $\C^n$ and $\L u(n)$. If $n \geq 3$, then every continuous Lie quasi-state on $\L u(n)$ is linear by Gleason's theorem, and Lie quasi-states on $\C^n$ are linear anyway. We deduce:
\begin{theorem}\label{unStrongRigidity} For $n \geq 3$ every continuous Lie quasi-state on $\gfr_n$ is linear.
\end{theorem}
Note that the theorem does not hold for $n \in \{1, 2\}$. In both cases the space of continuos Lie quasi-states on $\gfr_n$ is infinite-dimensional. For $n=1$ this was established in Subsection \ref{SolvableExample}, whereas for $n=2$ it follows from the fact that $\mathcal Q(\L u(2))$ is infinite-dimensional.

\subsection{Rigidity of unitary motion algebras}
Theorem \ref{unStrongRigidity} implies in particular that the Lie algebras $\L g_n$ are rigid for all $n \geq 3$. Moreover, we have already established rigidity for $n=1$ in Subsection \ref{SolvableExample}. The following theorem deals with the remaining case $n=2$.
\begin{theorem}\label{MotionAd}
The Lie algebra $\L g_n = \C^n \rtimes \L u(n)$ is rigid for all $n \geq 1$. Moreover, $\dim {\mathcal Q}_{\rm Ad}(\gfr_n)=1$ for all $n \geq 1$ and every continuous ${\rm Ad}$-invariant Lie quasi-state on $\L g_n$ is of the form
\[
\zeta(v,X) = i\lambda \cdot {\rm tr}(X)
\]
for some $\lambda \in \R$.
\end{theorem}
\begin{proof} Observe first that every Lie algebra homomorphism (i.e. every Ad-invariant \emph{linear} Lie quasi-state) $\gfr_n \to \R$ factors through $\L u(n)$, and every Lie algebra homomorphism $\L u(n) \to \R$ is of the form $X \mapsto  i\lambda \cdot {\rm tr}(X)$ for some $\lambda \in \R$. The latter is obvious for $n=1$ and follows from the decomposition $\L u(n) = \L{su}(n) \oplus \R \cdot i {\bf 1}$ for $n \geq 2$, since then $\L{su}(n)$ is simple. We conclude that the second statement of the theorem follows from the rigidity statement. In view of our previous results it thus only remains to establish rigidity for $\gfr_2$; the following argument works actually for all $n \geq 2$.

Let $n \geq 2$ and $\zeta: \gfr_n \to \R$ be a continuous Ad-invariant Lie quasi-state. By Subsection \ref{UMADec} there exist $\psi \in (\C^n)^*$ and $\zeta_{\L u(n)} \in \mathcal Q(\L u(n))$ such that
\[
\zeta(v,X) = \psi(v) + \zeta_{\L u(n)}(X).
\]
Now since $\zeta$ is ${\rm Ad}$-invariant, we have 
\[
\zeta(v,X) = \zeta({\rm Ad}(k)(v, X)) = \zeta(k.v, {\rm Ad}(k)(X))
\]
for all $(v,X) \in \gfr_n$ and all $k \in U(n)$, which we can write out as 
\[
\psi(v) + \zeta_{\L u(n)}(X) = \psi(k.v) +\zeta_{\L u(n)}({\rm Ad}(k)(X)).
\]
Setting $X = 0$, respectively $v=0$, we obtain that $\psi$ is radial and that $\zeta_{\L u(n)}$ is an Ad-invariant Lie quasi-state. Now every radial linear functional is trivial, and since $\L u(n)$ is reductive, we deduce from Theorem \ref{MainTheorem1} that every Ad-invariant Lie quasi-state on $\L u(n)$ is linear. It follows that $\zeta$ itself is linear, and thus $\gfr_n$ is rigid for $n \geq 2$.
\end{proof}

\newpage 
\appendix
\section{Generalized frame functions}\label{AppendixFrame}
In our study of Lie quasi-states on unitary motion algebras we introduced the following class of functions, which might be of independent interest:
\begin{definition} A function $f: \C^n \to \R$ is called a \emph{generalized frame function} if it satisfies $f(0) = 0$ and for every pair $(u,v)$ of orthogonal vectors in $\C^n \setminus\{0\}$ we have
\[
f(u+v) = f(u)+f(v).
\]
\end{definition}
If $n=1$, then every function $f: \C \to \C$ with $f(0) = 0$ is a generalized frame function, so the concept is meaningless. We will thus assume from now on that $n \geq 2$. In this case the notion of generalized frame function is rather restrictive.  Clearly, every homomorphism $f:\C^n \to \R$ is a generalized frame function. An example of a non-linear generalized frame function is given by $f(v) = \|v\|^2$. Indeed, this is the content of the Pythagoras theorem. In particular, generalized frame functions of linear and quadratic growth type exist in every dimension. The purpose of this appendix is to show that for $n \geq 2$ generalized frame functions of sublinear growth cannot exist. Here a function $f: \C^n \to \R$ is called \emph{sublinear} provided
\[
\lim_{v \to \infty} \frac{f(v)}{\|v\|} = 0.
\]
\begin{theorem}\label{FFTrivial} Let $n \geq 2$. Then every sublinear generalized frame function is trivial.
\end{theorem}
The proof of the theorem will be based on the following fundamental identity, which follows directly from the definition:
\begin{lemma}\label{BasicIdentityFrames}
Let  $f: \C^n \to \R$ be a generalized frame function. If $u,v \in \C^n \setminus\{0\}$ are orthogonal vectors of the same length, then
\[
f(u) = 2\cdot f\left(\frac{u}{2}\right)+f\left(\frac{v}{2}\right)+f\left(\frac{-v}{2}\right).
\]
\end{lemma}
\begin{proof} If $u$ and $v$ are orthogonal, then so are $u/2$ and $\pm v/2$. If, moreover, $u$ and $v$ have the same length, then also $(u+v)/2$ and $(u-v)/2$ are orthogonal. Thus using repeatedly the defining property of a generalized frame function we obtain
\[
f(u) = f\left(\frac{u+v}{2}+\frac{u-v}{2}\right) = f\left(\frac{u+v}{2}\right)+f\left(\frac{u-v}{2}\right)=f\left(\frac{u}{2}\right)+f\left(\frac{v}{2}\right)+f\left(\frac{u}{2}\right)+f\left(\frac{-v}{2}\right).
\]
\end{proof}
We will use the following terminology concerning functions $f: \C^n \to \R$. We say that $f$ is \emph{even} (respectively \emph{odd}) if $f(-v) =f(v)$ (respectively $f(-v) = -f(v)$) for all $v \in \C^n$, and that $f$ is \emph{radial} if $f(u) = f(v)$ for all $u,v \in \C^n$ with $\|u\|=\|v\|$. Moreover, we say that $f$ is \emph{weakly radial} if $f(u) = f(v)$ for all orthogonal $u,v \in \C^n$ with $\|u\|=\|v\|$. 

For $n \geq 3$ every weakly radial function is actually radial. Indeed, if $u,v \in \C^n$ are of the same length and $n \geq 3$, then there exists $w \in \C^n$ of the same length which is orthogonal to both of them, and thus $f(u) = f(w) = f(v)$. However, for $n\leq2$ the two concepts are different. In fact, for $n=1$ any function is weakly radial.

\begin{proof}[Proof of Theorem \ref{FFTrivial}] Throughout the proof let $n \geq 2$ and let $f: \C^n \to \R$ be a sublinear generalized frame function. We will show that $f$ is trivial by applying Lemma \ref{BasicIdentityFrames} three times.

\textsc{Claim 1:} If $f$ is odd, then $f\equiv 0$.

Assume that $f$ is odd and let $u \in \C^n \setminus\{0\}$. Since $n \geq 2$ we can find $v \in \C^n \setminus \{0\}$, which is orthogonal to $u$ and of the same length. Since $f$ is odd, we deduce from Lemma \ref{BasicIdentityFrames} that $f(u) = 2\cdot f(u/2)$. We deduce that for every $n \in \mathbb N$,
\[
f(u) = \frac{f(2^n u)}{2^n}.
\]
Now by sublinearity of $f$ the right hand side tends to $0$ as $n \to \infty$, whence $f(u) = 0$, showing that every odd sublinear generalized frame function is trivial.

\textsc{Claim 2:} $f$ is even.

Define $\widetilde f(u) := f(u) -f(-u)$. We clearly have $f(0) = 0$ and if $u, v\in \C^n \setminus \{0\}$ are orthogonal, then so are $-u$ and $-v$, whence we have 
\[
\widetilde{f}(u+v) = f(u+v)-f(-u-v) = f(u)+f(v) -f(-u) -f(-v) = \widetilde{f}(u) +\widetilde{f}(v),\]
i.e. $\widetilde{f}$ is a generalized frame function as well. Now sublinearity of $f$ implies sublinearity of $\widetilde{f}$, and $\widetilde{f}$ is odd by construction. We deduce from \textsc{Claim 1} that $\widetilde{f} = 0$, i.e. $f$ is even.

\textsc{Claim 3:} $f$ is weakly radial.

We know from \textsc{Claim 2} that $f$ is even. It thus follows from Lemma \ref{BasicIdentityFrames} that for every pair $(u,v)$ of orthogonal vectors in $\C^n \setminus\{0\}$ of the same length we have
\[
f(u) = 2\cdot f(u/2) + 2\cdot f(v/2).
\]
However, the right hand side is symmetric in $u$ and $v$, whence we obtain
\[
f(u) = 2\cdot f(u/2) + 2\cdot f(v/2) = 2\cdot f(v/2) + 2\cdot f(u/2) = f(v).
\]
This shows that $f$ is weakly radial.

\textsc{Claim 4:} $f$ is trivial.

Since $f(0) = 0$ it suffices to show that $f(u) = 0$ for every $u \in \C^n \setminus\{0\}$. Given such $u$ we can always find $v\in\C^n \setminus\{0\}$ of the same length as $u$ and orthogonal to $u$. Since $f$ is weakly radial we have 
\[f(u/2) = f(v/2).\]
On the other hand, as we have seen already in the proof of \textsc{Claim 3}, the fact that $f$ is even together with Lemma \ref{BasicIdentityFrames} implies
\[
f(u) = 2\cdot f(u/2) + 2\cdot f(v/2).
\]
Combining both identities we obtain
\[
f(u) = 4 \cdot f(u/2),
\]
which in turn implies
\[
f(2^n u) = 4 \cdot f(2^{n-1}u) = \dots = 4^n f(u), 
\]
or equivalently,
\[
f(u) = 2^{-n}\|u\|\cdot \frac{f(2^nu)}{\|2^nu\|}
\]
Now by sublinearity of $f$, both factors on the right hand side converge to $0$ as $n \to \infty$, showing that $f$ is trivial.
\end{proof}

\end{document}